\documentclass[11pt]{amsart}
\usepackage[letterpaper,margin=1.2in]{geometry}
\usepackage{amsbsy,amsfonts,amsmath,amssymb,amscd,amsthm,mathrsfs}
\usepackage{subfigure,enumitem,graphicx,epsfig,color}
\usepackage[colorlinks=true,linkcolor=blue,citecolor=green]{hyperref}
\normalsize

\newtheorem{theorem}{Theorem}[section]
\newtheorem{lemma}[theorem]{Lemma}

\newtheorem{definition}[theorem]{Definition}

\newtheorem{remark}[theorem]{Remark}

\newtheorem{proposition}[theorem]{Proposition}

\newtheorem*{theorem A}{Theorem A}
\newtheorem*{corollary B}{Corollary B}
\newtheorem*{theorem C}{Theorem C}
\newtheorem*{corollary D}{Corollary D}
\theoremstyle{definition}

\begin{document}
\title[Variational principle for topological pressure and entropy on subsets of NDSs]{Variational principles on subsets\\
 of non-autonomous dynamical systems:\\
  topological pressure and topological entropy}
\author[J. Nazarian Sarkooh]{Javad Nazarian Sarkooh$^{*}$}
\address{Department of Mathematics, Ferdowsi University of Mashhad, Mashhad, IRAN.}
\email{\textcolor[rgb]{0.00,0.00,0.84}{javad.nazariansarkooh@gmail.com}}
 \thanks{$^*$Corresponding author}
\begin{abstract}
This paper discusses the variational principles on subsets for topological pressure and topological entropy
of non-autonomous dynamical systems. Let $(X, f_{1,\infty})$ be a non-autonomous dynamical system and $\psi$ be a continuous potential on $X$, where $(X,d)$ is a compact metric space and $f_{1,\infty}=(f_n)_{n=1}^\infty$ is a sequence of continuous maps $f_n: X\to X$. We define the Pesin-Pitskel topological pressure (weighted topological pressure) $P_{f_{1,\infty}}^{B}(Z,\psi)$ ($P_{f_{1,\infty}}^{\mathcal{W}}(Z,\psi)$) and the Bowen topological entropy (weighted Bowen topological entropy) $h_{top}^{B}(f_{1,\infty},Z)$ ($h_{top}^{WB}(f_{1,\infty},Z)$) for any subset $Z$ of $X$. Also, we define the measure-theoretic pressure $P_{\mu,f_{1,\infty}}(X,\psi)$ and the measure-theoretic lower entropy $\underline{h}_{\mu}(f_{1,\infty})$ for any $\mu\in\mathcal{M}(X)$, where $\mathcal{M}(X)$ denotes the set of all Borel probability measures on $X$. Then, for any nonempty compact subset $Z$ of $X$, we show the following variational principles for topological pressure and topological entropy
\begin{equation*}
P_{f_{1,\infty}}^{B}(Z,\psi)=P_{f_{1,\infty}}^{\mathcal{W}}(Z,\psi)=\sup\{P_{\mu,f_{1,\infty}}(X,\psi):\mu\in\mathcal{M}(X), \mu(Z)=1\},
\end{equation*}
\begin{equation*}
h_{top}^{B}(f_{1,\infty},Z)=h_{top}^{WB}(f_{1,\infty},Z)=\sup\{\underline{h}_{\mu}(f_{1,\infty}):\mu\in\mathcal{M}(X), \mu(Z)=1\}.
\end{equation*}
Moreover, we show that the Pesin-Pitskel topological pressure (weighted topological pressure) and the Bowen topological entropy (weighted Bowen topological entropy) can be determined by the measure-theoretic pressure and the measure-theoretic lower entropy of Borel probability measures, respectively. These results extend Feng and Huang's results on entropies \cite{DFWH}, Ma and Wen's results on entropies \cite{MJHWZY}, and Tang et al. results on pressures \cite{TXCWCZY} for classical dynamical systems to pressures and entropies of non-autonomous dynamical systems. 
\end{abstract}

\maketitle
\thispagestyle{empty}

\section{Introduction}
Throughout this paper $X$ is a compact metric space with metric $d$ and $f_{1,\infty}=(f_n)_{n=1}^\infty$ is a sequence of continuous maps $f_n: X\to X$, such a pair $(X, f_{1,\infty})$ is a non-autonomous dynamical system so-called time-dependent (NDS for short). The non-autonomous systems yield very flexible models than autonomous cases for the study and description of real-world processes. They may be used to describe the evolution of a wider class of phenomena, including systems that are forced or driven. Recently, there have been major efforts in establishing a general theory of such systems (see \cite{BV,JNSFHG1,K1,K2,K3,KR,KS,JNSFHG,JNS000,O,DTRD}), but a global theory is still out of reach. Our main goal in this paper is to describe the topological and measure-theoretical aspects of the thermodynamic formalism for non-autonomous dynamical systems.
Thermodynamic formalism, that is the formalism of equilibrium statistical physics, was adapted to the theory of dynamical systems in the classical works of Sinai, Ruelle, and Bowen in the 1970s \cite{RB5,DR000,DR111,DR222}. Topological pressure, topological entropy, Gibbs measures, and equilibrium states are the fundamental notions in thermodynamic formalism. 
Topological pressure is the main tool in studying the dimension of invariant sets and measures for dynamical systems in dimension theory. The notion of entropy, on the other hand, is one of the most important objects in dynamical systems, either as a topological invariant or as a measure of the chaoticity of dynamical systems. For uniformly hyperbolic systems Bowen \cite{RB5} presented a complete description of thermodynamic formalism, but in the non-uniformly hyperbolic case a general theory of thermodynamic formalism, despite substantial progress by several authors, is far from being complete.

In the theory of dynamical systems, entropies are nonnegative extended real numbers measuring the complexity of a dynamical system. Topological entropy was first introduced by Adler et al. \cite{AKM} via open covers for continuous maps in compact topological spaces. In 1970, Bowen \cite{RB} gave another definition in terms of separated and spanning sets for uniformly continuous maps in metric spaces, and this definition is equivalent to Adler's definition for continuous maps in compact metric spaces. Later, for noncompact sets, Bowen \cite{RB4} gave a characterization of dimension type for the entropy, which was further investigated by Pesin and Pitskel \cite{PYBPBS}. Also, the measure-theoretic (metric) entropy for an invariant measure was introduced by Kolmogorov \cite{KAAAA}. The basic relation between topological entropy and measure-theoretic entropy is the variational principle, see \cite{PW}. Note that, topological entropy has close relationships with many important dynamical properties, such as chaos, Lyapunov exponents, the growth of the number of periodic points, and so on. 

In 1996, Kolyada and Snoha extended the concept of topological entropy to non-autonomous systems, based on open covers, separated sets, and spanning sets, and obtained a series of important properties of these systems \cite{KS}. Recently, Kawan \cite{K1,K2,K3} introduced and studied the notion of metric entropy for
non-autonomous dynamical systems and showed that it is related via variational inequality (principle) to the topological entropy as defined by Kolyada and Snoha. More precisely, for equicontinuous topological
non-autonomous dynamical systems
$(X_{1,\infty},f_{1,\infty})$, he proved the variational inequality
\begin{equation*}
\sup_{\mu_{1,\infty}}h_{\mathcal{E}_{M}}(f_{1,\infty},\mu_{1,\infty})\leq h_{\text{top}}(f_{1,\infty}),
\end{equation*}
where $h_{\text{top}}(f_{1,\infty})$ and $h_{\mathcal{E}_{M}}(f_{1,\infty},\mu_{1,\infty})$ are the topological
and metric entropy of $(X_{1,\infty},f_{1,\infty})$, the supremum is taken over all invariant measure sequences $\mu_{1,\infty}$ and $\mathcal{E}_{M}$ is Misiurewicz class of partitions. Also, for non-autonomous dynamical systems $(M,f_{1,\infty})$ built from $C^{1}$ expanding maps $f_{n}$ on a compact Riemannian manifold $M$ with uniform bounds on expansion factors and derivatives that act in the same way on the fundamental group of $M$, he proved the full variational principle
\begin{equation*}
\sup_{\mu_{1,\infty}}h_{\mathcal{E}_{M}}(f_{1,\infty},\mu_{1,\infty})=h_{\text{top}}(f_{1,\infty}).
\end{equation*}

Topological pressure, as a non-trivial and natural generalization of topological entropy, roughly measures the orbit complexity of the iterated map on the potential function and is one of the most fundamental dynamical invariants associated with a continuous map and a potential function. 
Ruelle \cite{DR} introduced topological pressure of a continuous function for $\mathbb{Z}^{n}$-actions on compact spaces and established the variational principle for topological pressure and then Walters \cite{PW1} generalized the variational principle for a $\mathbb{Z}^{+}$-action without these assumptions. Misiurewicz \cite{MMA} gave an elegant proof of the variational principle for $\mathbb{Z}_{+}^{n}$-actions. From a viewpoint of dimension theory, Pesin and Pitskel \cite{PYBPBS} defined the topological pressure on noncompact sets which is a generalization of Bowen's definition of topological entropy on noncompact sets \cite{RB4}, and they proved the variational principle under some supplementary conditions. Moreover, the variational principle for topological pressure was obtained by so many authors in other situations, see \cite{BL,BL1,CYFDHW,HWYXZG,HWYY,MMA,MAAAAAA,OJMPD,ZYCW,ZGGGG}.
Note that the notions of the topological pressure, topological entropy, variational principle, equilibrium states, and Gibbs measures that are the fundamental notions in thermodynamic formalism play a fundamental role in statistical mechanics, ergodic theory, and dynamical systems, see \cite{RB4,VO,PW}.

In 2008, Huang et al. \cite{HWZ} extended the notion of topological pressure to non-autonomous
dynamical systems, by using open covers, separated sets, and spanning sets. Recently, Kawan \cite{K2} introduced the notion of measure-theoretic pressure for non-autonomous dynamical systems and show that it is related via a variational inequality to the topological pressure. Moreover, Nazarian Sarkooh and Ghane \cite{JNSFHG} studied the regularity of the topological pressure function. Nevertheless, there is nothing provided any full variational principle of topological pressure for non-autonomous dynamical systems.

Motivated by Feng and Huang \cite{DFWH} and Ma and Wen \cite{MJHWZY}, where the authors studied the relations between the Bowen topological entropy on an arbitrary subset and the measure-theoretic lower entropy of Borel probability measures in classical dynamical systems, we define and study the Bowen topological entropy (weighted Bowen topological entropy) on an arbitrary subset and the measure-theoretic lower entropy of a Borel probability measure for non-autonomous dynamical systems. Then, we prove a Variational principle for topological entropy which links the Bowen topological entropy (weighted Bowen topological entropy) on an arbitrary nonempty compact subset to the measure-theoretic lower entropy of Borel probability measures for non-autonomous dynamical systems. Moreover, we show that the Bowen topological entropy (weighted Bowen topological entropy) is determined by the measure-theoretic lower entropy of Borel probability measures.
Also, motivated by Tang et al. recent work \cite{TXCWCZY}, where the authors studied the relations between the
Pesin-Pitskel topological pressure on an arbitrary subset and the measure-theoretic pressure of Borel probability measures in classical dynamical systems, we define and study the Pesin-Pitskel topological pressure (weighted topological pressure) on an arbitrary subset and the measure-theoretic pressure of a Borel probability measure for non-autonomous dynamical systems. Then, we prove a Variational principle for topological pressure which links the Pesin-Pitskel topological pressure (weighted topological pressure) on an arbitrary nonempty compact subset to the measure-theoretic pressure of Borel probability measures for non-autonomous dynamical systems. Moreover, we show that Pesin-Pitskel topological pressure (weighted topological pressure) is determined by measure-theoretic pressure of Borel probability measures.
 
\textbf{This paper is organized as follows.} In Section \ref{section2}, we give a precise definition of a
non-autonomous dynamical system and review the main concepts and main results. The proof of the main results and related propositions are given in Sections \ref{section3} and \ref{section4}.
\section{Definitions and the statement of main results}\label{section2}
A \emph{non-autonomous} or \emph{time-dependent} dynamical system (an \emph{NDS} for short), is a pair $(X_{1,\infty}, f_{1,\infty})$, where $X_{1,\infty}=(X_n)_{n=1}^\infty$ is a sequence of sets
and $f_{1,\infty}=(f_n)_{n=1}^\infty$ is a sequence of maps $f_n: X_n \to X_{n+1}$.
If all the sets $X_n$ are compact metric spaces and all the $f_n$ are continuous,
we say that $(X_{1,\infty}, f_{1,\infty})$ is a \emph{topological NDS}.
Here, we assume that $X$ is a compact metric space with metric $d$, all the sets $X_{n}$ are equal to the set $X$ and we abbreviate $(X_{1,\infty}, f_{1,\infty})$ by $(X, f_{1,\infty})$. Throughout this paper, we work with topological NDSs and use NDS instead of topological NDS for simplicity.
The time evolution of the system is defined by composing the maps $f_{n}$ in an obvious way.
In general, we define
\begin{equation*}
f_i^n:=f_{i+n-1}\circ\cdots\circ f_{i+1}\circ f_i \ \ \text{for} \ i,n\in \mathbb{N},\ \text{and} \ f_i^0:=\text{id}_X.
\end{equation*}
We also put $f_i^{-n}:=(f_i^n)^{-1}$, which is only applied to subsets $A \subset X$.
The \emph{trajectory} of a point $x \in X$ is the sequence $(f_1^n(x))_{n=0}^\infty$.

Fix an NDS $(X, f_{1,\infty})$ on a compact metric space $(X,d)$, then the \emph{Bowen-metrics} on $X$ are given by
\begin{equation*}\label{eqq1}
d_{i,n}(x,y):=\max_{0\leq j<n}d(f_{i}^{j}(x),f_{i}^{j}(y))\ \text{for}\ i,n\geq 1\ \text{and}\ x,y\in X.
\end{equation*}
Also, for any $i,n\geq1$, $x\in X$, and $\epsilon >0$, we define
\begin{equation*}\label{eqq2}
B(x,i,n,\epsilon):=\{y \in X: d_{i,n}(x,y)<\epsilon\},
\end{equation*}
which is called a \emph{dynamical} $n$-\emph{ball} with initial time $i$. In particular, denote $B(x,1,n,\epsilon)$ and $d_{1,n}$ by $B_{n}(x,\epsilon)$ and $d_{n}$, respectively.

Let $\mathcal{C}(X,\mathbb{R})$ denote the Banach space of all continuous real-valued functions on $X$ equipped with the supremum norm, then for $\psi\in\mathcal{C}(X,\mathbb{R})$ and $i,n\in\mathbb{N}$, denote
$\Sigma_{j=0}^{n-1}\psi(f_{i}^{j}(x))$ by $S_{i,n}\psi(x)$. Also, for subset $U$ of $X$,
put $S_{i,n}\psi(U):=\sup_{x\in U}S_{i,n}\psi(x)$.  In particular, let
$ S_{i,n}\psi(x,\epsilon):=S_{i,n}\psi(B(x,i,n,\epsilon)) $.
\subsection{Pesin-Pitskel topological pressure and Bowen topological entropy}
Let $(X, f_{1,\infty})$ be an NDS on a compact metric space $(X,d)$ and $Z$ be a nonempty subset of $X$. Given $n\in\mathbb{N}$, $\alpha\in\mathbb{R}$, $\epsilon>0$, and $\psi\in\mathcal{C}(X,\mathbb{R})$, define
\begin{equation}\label{eqq3}
M_{f_{1,\infty}}(n,\alpha,\epsilon,Z,\psi):=\inf\bigg\{\sum_{i}\text{e}^{-\alpha n_{i}+S_{1,n_{i}}\psi(x_{i},\epsilon)}:Z\subseteq\bigcup_{i}B_{n_{i}}(x_{i},\epsilon)\bigg\},
\end{equation}
where the infimum is taken over all finite or countable collections of $\{B_{n_{i}}(x_{i},\epsilon)\}_{i}$ such that
$x_{i}\in X$, $n_{i}\geq n$, and $Z\subseteq\bigcup_{i}B_{n_{i}}(x_{i},\epsilon)$. The quantity $M_{f_{1,\infty}}(n,\alpha,\epsilon,Z,\psi)$ is monotone in $n$, hence the following limit exists:
\begin{equation}\label{eqq11}
M_{f_{1,\infty}}(\alpha,\epsilon,Z,\psi):=\lim_{n\to\infty}M_{f_{1,\infty}}(n,\alpha,\epsilon,Z,\psi).
\end{equation}
By the construction of Carath\'{e}odory dimension characteristics (see \cite{PYB}), when $\alpha$ goes from
$-\infty$ to $\infty$, the quantity $M_{f_{1,\infty}}(\alpha,\epsilon,Z,\psi)$ 
 jumps from $\infty$ to $0$ at a unique critical value. 
Hence we can define the number
\begin{equation*}
P_{f_{1,\infty}}^{B}(\epsilon,Z,\psi):=\sup\{\alpha:M_{f_{1,\infty}}(\alpha,\epsilon,Z,\psi)=\infty\}=\inf\{\alpha:M_{f_{1,\infty}}(\alpha,\epsilon,Z,\psi)=0\}.
\end{equation*}
\begin{definition}[Pesin-Pitskel topological pressure]\label{deff1}
Let $(X, f_{1,\infty})$ be an NDS on a compact metric space $(X,d)$, $Z$ be a nonempty subset of $X$ and $\psi\in\mathcal{C}(X,\mathbb{R})$. Then, the \emph{Pesin-Pitskel topological pressure} of NDS $(X,f_{1,\infty})$ on the set $Z$ with respect to $\psi$ is defined as

\begin{equation*}
P_{f_{1,\infty}}^{B}(Z,\psi):=\lim_{\epsilon\to 0} P_{f_{1,\infty}}^{B}(\epsilon,Z,\psi).
\end{equation*}
\end{definition}
\begin{definition}[Bowen topological entropy]
In Definition \ref{deff1}, 
if $\psi=0$ then we have the definition of \emph{Bowen topological entropy} of NDS $(X,f_{1,\infty})$ on the set $Z$, which we denote by $h_{top}^{B}(f_{1,\infty},Z)$, i.e.,
\begin{equation*}
h_{top}^{B}(f_{1,\infty},Z):=P_{f_{1,\infty}}^{B}(Z,0).
\end{equation*}
\end{definition}
The Pesin-Pitskel topological pressure and Bowen topological entropy can be defined in the following alternative way, see \cite{BL} or \cite{PYB} for more details. Let $(X, f_{1,\infty})$ be an NDS on a compact metric space $(X,d)$, $\psi\in\mathcal{C}(X,\mathbb{R})$, and $\mathcal{U}$ be a finite open cover of $X$. Denote the diameter of the open cover by $|\mathcal{U}|:=\max\{\text{diam}(U): U\in\mathcal{U}\}$. For $n\in\mathbb{N}$ we denote 
by $\mathcal{W}_{n}(\mathcal{U})$ the collection of strings $ \textbf{U}=U_{1}U_{2}\ldots U_{n} $ with $ U_{i}\in\mathcal{U} $. For $\textbf{U}\in\mathcal{W}_{n}(\mathcal{U})$
we call the integer $ m(\textbf{U})=n $ the length of $ \textbf{U} $ and define
\begin{equation*}
X(\textbf{U})=U_{1}\cap f_{1}^{-1}(U_{2})\cap\cdots\cap f_{1}^{-(n-1)}(U_{n})=\{x\in X: f_{1}^{j-1}(x)\in U_{j}\ \text{for}\ j=1,\ldots,n\}.
\end{equation*}
Let $Z\subseteq X$. We say that $ \Lambda\subset\bigcup_{n\geq 1}\mathcal{W}_{n}(\mathcal{U}) $ covers $ Z $ if $ \bigcup_{\textbf{U}\in\Lambda}X(\textbf{U})\supset Z $. For $ \alpha\in\mathbb{R} $ and $N\in\mathbb{N}$, define
\begin{equation*}
M_{f_{1,\infty}}^{\alpha}(\mathcal{U},N,Z,\psi):=\inf_{\Lambda}\bigg\{\sum_{\textbf{U}\in\Lambda}\text{e}^{-\alpha m(\textbf{U})+S_{1,m(\textbf{U})}\psi(X(\textbf{U}))}\bigg\},
\end{equation*}
where the infimum is taken over all $ \Lambda\subset\bigcup_{n\geq N}\mathcal{W}_{n}(\mathcal{U}) $ that cover $ Z $ and $ S_{1,m(\textbf{U})}\psi(X(\textbf{U}))=-\infty $ if $ X(\textbf{U})=\emptyset $. Clearly $ M_{f_{1,\infty}}^{\alpha}(\mathcal{U},N,.,\psi) $ is a finite outer measure on $ X $, and
\begin{equation*}
M_{f_{1,\infty}}^{\alpha}(\mathcal{U},N,Z,\psi)=\inf\{M_{f_{1,\infty}}^{\alpha}(\mathcal{U},N,G,\psi): G\supset Z\ \text{and}\ G\ \text{is open}\}.
\end{equation*}
Note that the quantity $ M_{f_{1,\infty}}^{\alpha}(\mathcal{U},N,Z,\psi) $ increases as $ N $ increases. Define
\begin{equation*}
M_{f_{1,\infty}}^{\alpha}(\mathcal{U},Z,\psi):=\lim_{N\to\infty}M_{f_{1,\infty}}^{\alpha}(\mathcal{U},N,Z,\psi)
\end{equation*}
and
\begin{equation*}
P_{f_{1,\infty}}^{B}(\mathcal{U},Z,\psi):=\sup\{\alpha:M_{f_{1,\infty}}^{\alpha}(\mathcal{U},Z,\psi)=\infty\}=\inf\{\alpha:M_{f_{1,\infty}}^{\alpha}(\mathcal{U},Z,\psi)=0\}.
\end{equation*}
Then
\begin{equation*}
P_{f_{1,\infty}}^{B}(Z,\psi)=\sup_{\mathcal{U}}P_{f_{1,\infty}}^{B}(\mathcal{U},Z,\psi)\ \ \text{and}\ \
h_{top}^{B}(f_{1,\infty},Z)=\sup_{\mathcal{U}}P_{f_{1,\infty}}^{B}(\mathcal{U},Z,0),
\end{equation*}
where $\mathcal{U}$ runs over finite open covers of $ Z $. It is known that
\begin{equation*}
\sup_{\mathcal{U}}P_{f_{1,\infty}}^{B}(\mathcal{U},Z,\psi)=\lim_{|\mathcal{U}|\to 0}P_{f_{1,\infty}}^{B}(\mathcal{U},Z,\psi).
\end{equation*}
\begin{remark}\label{remarkk1}
If we replace $S_{1,n_{i}}\psi(x_{i},\epsilon)$ 
by $S_{1,n_{i}}\psi(x_{i})$ 
in equation (\ref{eqq3}), 
then we can define new functions $\mathcal{M}_{f_{1,\infty}}(n,\alpha,\epsilon,Z,\psi)$ and
$\mathcal{M}_{f_{1,\infty}}(\alpha,\epsilon,Z,\psi)$
as (\ref{eqq3}) and
(\ref{eqq11}).
Moreover, we denote the critical value of
$\mathcal{M}_{f_{1,\infty}}(\alpha,\epsilon,Z,\psi)$
by
$\mathcal{P}_{f_{1,\infty}}^{B}(\epsilon,Z,\psi)$.
\end{remark}
\subsection{Weighted topological pressure and weighted Bowen topological entropy}
Let $(X, f_{1,\infty})$ be an NDS on a compact metric space $(X,d)$, $Z$ be a nonempty subset of $X$ and $\psi\in\mathcal{C}(X,\mathbb{R})$. For any $n\in\mathbb{N}$, $\alpha\in\mathbb{R}$, $\epsilon>0$, and any bounded function $g:X\to\mathbb{R}$, define
\begin{equation}\label{eqq7}
\mathcal{W}_{f_{1,\infty}}(n,\alpha,\epsilon,g,\psi):=\inf\bigg\{\sum_{i}c_{i}\text{e}^{-\alpha n_{i}+S_{1,n_{i}}\psi(x_{i},\epsilon)}\bigg\},
\end{equation}
where the infimum is taken over all finite or countable families of $\{B_{n_{i}}(x_{i},\epsilon),c_{i}\}$ such that
$0<c_{i}<\infty$, $x_{i}\in X$, $n_{i}\geq n$ for all $i$, and $\sum_{i}c_{i}\chi_{B_{i}}\geq g$,
where $B_{i}:=B_{n_{i}}(x_{i},\epsilon)$ and $\chi_{B}$ denotes the characteristic function on a subset $B$ of $X$. 
For $Z\subseteq X$, let $g=\chi_{Z}$ and
\begin{equation*}
\mathcal{W}_{f_{1,\infty}}(n,\alpha,\epsilon,Z,\psi):=\mathcal{W}_{f_{1,\infty}}(n,\alpha,\epsilon,\chi_{Z},\psi).
\end{equation*}
The quantity $\mathcal{W}_{f_{1,\infty}}(n,\alpha,\epsilon,Z,\psi)$ does not decrease as $n$ increases, hence the following limit exists:
\begin{equation*}
\mathcal{W}_{f_{1,\infty}}(\alpha,\epsilon,Z,\psi):=\lim_{n\to\infty}\mathcal{W}_{f_{1,\infty}}(n,\alpha,\epsilon,Z,\psi).
\end{equation*}
Clearly, there exists a critical value of the parameter $\alpha$,
which we denote by $P_{f_{1,\infty}}^{\mathcal{W}}(\epsilon,Z,\psi)$, where $\mathcal{W}_{f_{1,\infty}}(\alpha,\epsilon,Z,\psi)$ jumps from $\infty$ to $0$, i.e.,
\begin{equation*}
\mathcal{W}_{f_{1,\infty}}(\alpha,\epsilon,Z,\psi) = \left\{
\begin{array}{rl}
\infty & \text{if } \alpha<P_{f_{1,\infty}}^{\mathcal{W}}(\epsilon,Z,\psi),\\
0 & \text{if } \alpha>P_{f_{1,\infty}}^{\mathcal{W}}(\epsilon,Z,\psi).
\end{array} \right.
\end{equation*}
\begin{definition}[Weighted topological pressure]\label{defnew1}
Let $(X, f_{1,\infty})$ be an NDS on a compact metric space $(X,d)$, $Z$ be a subset of $X$ and $\psi\in\mathcal{C}(X,\mathbb{R})$. Since the quantity $P_{f_{1,\infty}}^{\mathcal{W}}(\epsilon,Z,\psi)$ is monotone with respect to $\epsilon$, the following limit exists:
\begin{equation*}
P_{f_{1,\infty}}^{\mathcal{W}}(Z,\psi):=\lim_{\epsilon\to 0}P_{f_{1,\infty}}^{\mathcal{W}}(\epsilon,Z,\psi).
\end{equation*}
We call $P_{f_{1,\infty}}^{\mathcal{W}}(Z,\psi)$ the \emph{weighted topological pressure} of NDS $(X,f_{1,\infty})$ on the set $Z$ with respect to $\psi$.
\end{definition}
\begin{definition}[Weighted Bowen topological entropy]
In Definition \ref{defnew1}, 
if $\psi=0$ then we have the definition of \emph{weighted Bowen topological entropy} of NDS $(X,f_{1,\infty})$ on the set $Z$, which we denote by $h_{top}^{WB}(f_{1,\infty},Z)$, i.e.,
\begin{equation*}
h_{top}^{WB}(f_{1,\infty},Z):=P_{f_{1,\infty}}^{W}(Z,0).
\end{equation*}
\end{definition}
\subsection{Measure-theoretic pressure and measure-theoretic lower entropy}
Let $(X, f_{1,\infty})$ be an NDS on a compact metric space $(X,d)$, and let $\mathcal{M}(X)$ denote the set of Borel probability measures on $X$. For $\mu\in\mathcal{M}(X)$, $x\in X$, and $\psi\in\mathcal{C}(X,\mathbb{R})$, we define
\begin{equation}\label{eqq5}
P_{\mu,f_{1,\infty}}(x,\psi):=\lim_{\epsilon\to 0}\liminf_{n\to\infty}\dfrac{-\log\mu(B_{n}(x,\epsilon))+S_{1,n}\psi(x)}{n}
\end{equation}
and
\begin{equation}\label{new1}
\underline{h}_{\mu}(f_{1,\infty},x):=\lim_{\epsilon\to 0}\liminf_{n\to\infty}\dfrac{-\log\mu(B_{n}(x,\epsilon))}{n}.
\end{equation}
\begin{definition}[Measure-theoretic pressure and measure-theoretic lower entropy]
Let $(X, f_{1,\infty})$ be an NDS on a compact metric space $(X,d)$, $\psi\in\mathcal{C}(X,\mathbb{R})$, and $\mu\in\mathcal{M}(X)$.
Then, the \emph{measure-theoretic pressure} of NDS $(X, f_{1,\infty})$ with
respect to $\mu$ and $\psi$ is defined as
\begin{equation}\label{eqq6}
P_{\mu,f_{1,\infty}}(X,\psi):=\int P_{\mu,f_{1,\infty}}(x,\psi)\ d\mu(x).
\end{equation}
Also, the \emph{measure-theoretic lower entropy} of NDS $(X, f_{1,\infty})$ with
respect to $\mu$ is defined as
\begin{equation}\label{new2}
\underline{h}_{\mu}(f_{1,\infty}):=\int \underline{h}_{\mu}(f_{1,\infty},x)\ d\mu(x).
\end{equation}
\end{definition}
\subsection{Statement of main results}
Now, we state the main results of this paper and postpone the proofs to the next sections. The first theorem shows that the Pesin-Pitskel topological pressure (weighted topological pressure) is determined by the measure-theoretic pressure of Borel probability measures, which extends \cite[Theorem A]{TXCWCZY} for NDSs.
\begin{theorem A}
Let $(X, f_{1,\infty})$ be an NDS on a compact metric space $(X,d)$, $\psi\in\mathcal{C}(X,\mathbb{R})$, $\mu\in\mathcal{M}(X)$ and $Z$ be a subset of $X$. For $s\in\mathbb{R}$, the following properties hold:
\begin{itemize}
\item[(1)]
If $P_{\mu,f_{1,\infty}}(x,\psi)\leq s$ for all $x\in Z$, then $P_{f_{1,\infty}}^{B}(Z,\psi)\leq s$ and 
$P_{f_{1,\infty}}^{\mathcal{W}}(Z,\psi)\leq s$;
\item[(2)]
If $P_{\mu,f_{1,\infty}}(x,\psi)\geq s$ for all $x\in Z$ and $\mu(Z)>0$, then $P_{f_{1,\infty}}^{B}(Z,\psi)\geq s$ and $P_{f_{1,\infty}}^{\mathcal{W}}(Z,\psi)\geq s$.
\end{itemize}
\end{theorem A}
As a direct consequence of Theorem A, we have the following corollary which shows that the Bowen topological entropy (weighted Bowen topological entropy) is determined by the measure-theoretic lower entropy of Borel probability measures, which extends \cite[Theorem 1]{MJHWZY} for NDSs.
\begin{corollary B}
Let $(X, f_{1,\infty})$ be an NDS on a compact metric space $(X,d)$, $\mu\in\mathcal{M}(X)$ and $Z$ be a subset of $X$. For $s\in\mathbb{R}$, the following properties hold:
\begin{itemize}
\item[(1)]
If $\underline{h}_{\mu}(f_{1,\infty},x)\leq s$ for all $x\in Z$, then $h_{top}^{B}(f_{1,\infty},Z)\leq s$ and $h_{top}^{WB}(f_{1,\infty},Z)\leq s$;
\item[(2)]
If $\underline{h}_{\mu}(f_{1,\infty},x)\geq s$ for all $x\in Z$ and $\mu(Z)>0$, then $h_{top}^{B}(f_{1,\infty},Z)\geq s$ and $h_{top}^{WB}(f_{1,\infty},Z)\geq s$.
\end{itemize}
\end{corollary B}
The next theorem that is the main result of this paper is a variational principle for topological pressure which links the Pesin-Pitskel topological pressure (weighted topological pressure) on an arbitrary nonempty compact subset to the measure-theoretic pressure of Borel probability measures, which extends \cite[Theorem B, part (1)]{TXCWCZY} for NDSs.

\begin{theorem C}
Let $(X, f_{1,\infty})$ be an NDS on a compact metric space $(X,d)$ and $Z$ be a nonempty compact subset of $X$. Then, for $\psi\in\mathcal{C}(X,\mathbb{R})$, 
\begin{equation*} 
P_{f_{1,\infty}}^{B}(Z,\psi)=P_{f_{1,\infty}}^{\mathcal{W}}(Z,\psi)=\sup\{P_{\mu,f_{1,\infty}}(X,\psi):\mu\in\mathcal{M}(X), \mu(Z)=1\}.
\end{equation*}
\end{theorem C}
As a direct consequence of Theorem C, we have the following corollary that is a variational principle for topological entropy which links the Bowen topological entropy (weighted Bowen topological entropy) on an arbitrary nonempty compact subset to the measure-theoretic lower entropy of Borel probability measures, which extends \cite[Theorem 1.2, part (1)]{DFWH} for NDSs.
\begin{corollary D}
Let $(X, f_{1,\infty})$ be an NDS on a compact metric space $(X,d)$ and $Z$ be a nonempty compact subset of $X$. Then, for $\psi\in\mathcal{C}(X,\mathbb{R})$, 
\begin{equation*}
h_{top}^{B}(f_{1,\infty},Z)=h_{top}^{WB}(f_{1,\infty},Z)=\sup\{\underline{h}_{\mu}(f_{1,\infty}):\mu\in\mathcal{M}(X), \mu(Z)=1\}.
\end{equation*}
\end{corollary D}
\section{Proof of theorem A and corollary B}\label{section3}
In this section, we give a few preparatory results needed for the proof of the main
results and then we prove Theorem A and Corollary B.

By Definition \ref{deff1} and Remark \ref{remarkk1}, we have the following theorem.
\begin{theorem}\label{theorem3}
Let $(X, f_{1,\infty})$ be an NDS on a compact metric space $(X,d)$, $Z$ be a subset of $X$ and $\psi\in\mathcal{C}(X,\mathbb{R})$. Then
\begin{equation*}
P_{f_{1,\infty}}^{B}(Z,\psi)=\lim_{\epsilon\to 0} P_{f_{1,\infty}}^{B}(\epsilon,Z,\psi)=\lim_{\epsilon\to 0} \ \ \mathcal{P}^{B}_{f_{1,\infty}}(\epsilon,Z,\psi).
\end{equation*}
\end{theorem}
\begin{proof}
Obviously, $P_{f_{1,\infty}}^{B}(Z,\psi)\geq\lim_{\epsilon\to 0} \mathcal{P}^{B}_{f_{1,\infty}}(\epsilon,Z,\psi)$. We shall prove the converse inequality. Hence, let
\begin{equation*}
\lambda(\epsilon):=\sup\{|\psi(x)-\psi(y)|: d(x,y)<2\epsilon\}.
\end{equation*}
Then, for any $n\in\mathbb{N}$, we have 
\begin{equation*}
S_{1,n}\psi(x,\epsilon)-n\lambda(\epsilon)<S_{1,n}\psi(x).
\end{equation*}
This implies that
\begin{eqnarray*}
\mathcal{M}_{f_{1,\infty}}(n,\alpha,\epsilon,Z,\psi)
&= &\inf\bigg\{\sum_{i}\text{e}^{-\alpha n_{i}+S_{1,n_{i}}\psi(x_{i})}:Z\subseteq\bigcup_{i}B_{n_{i}}(x_{i},\epsilon)\bigg\}\\
&\geq & \inf\bigg\{\sum_{i}\text{e}^{-\alpha n_{i}+S_{1,n_{i}}\psi(x_{i},\epsilon)-n_{i}\lambda(\epsilon)}:Z\subseteq\bigcup_{i}B_{n_{i}}(x_{i},\epsilon)\bigg\}\\
&= & \inf\bigg\{\sum_{i}\text{e}^{-(\alpha+\lambda(\epsilon)) n_{i}+S_{1,n_{i}}\psi(x_{i},\epsilon)}:Z\subseteq\bigcup_{i}B_{n_{i}}(x_{i},\epsilon)\bigg\}\\
&= & M_{f_{1,\infty}}(n,\alpha+\lambda(\epsilon),\epsilon,Z,\psi),
\end{eqnarray*}
where the infimums are taken over all finite or countable collections of $\{B_{n_{i}}(x_{i},\epsilon)\}_{i}$ such that
$x_{i}\in X$, $n_{i}\geq n$, and $Z\subseteq\bigcup_{i}B_{n_{i}}(x_{i},\epsilon)$. Hence
\begin{equation*}
\mathcal{M}_{f_{1,\infty}}(\alpha,\epsilon,Z,\psi)\geq M_{f_{1,\infty}}(\alpha+\lambda(\epsilon),\epsilon,Z,\psi).
\end{equation*}
It follows that
\begin{equation*}
\mathcal{P}^{B}_{f_{1,\infty}}(\epsilon,Z,\psi)\geq P_{f_{1,\infty}}^{B}(\epsilon,Z,\psi)-\lambda(\epsilon).
\end{equation*}
Now, we obtain the desired inequality as $\epsilon\to 0$, because $X$ is compact and so $\psi$ is uniformly continuous.
\end{proof}
By the construction of Carath\'{e}odory dimension characteristics (see \cite{BL,PYB}), we have the following proposition that collects some properties of the pressures and the entropies of NDSs which will be used in the proof of the main results.
\begin{proposition}\label{theorem2}
Let $(X, f_{1,\infty})$ be an NDS on a compact metric space $(X,d)$, $\psi\in\mathcal{C}(X,\mathbb{R})$ and 
$Z$, $Z_{1}$ and $Z_{2}$ be subsets of $X$. Then,
\begin{enumerate}
\item[(a)]
if $Z_{1}\subseteq Z_{2}$, then we have $P^{B}_{f_{1,\infty}}(Z_{1},\psi)\leq P^{B}_{f_{1,\infty}}(Z_{2},\psi)$, $P^{\mathcal{W}}_{f_{1,\infty}}(Z_{1},\psi)\leq P^{\mathcal{W}}_{f_{1,\infty}}(Z_{2},\psi)$, $h_{top}^{B}(f_{1,\infty},Z_{1})\leq h_{top}^{B}(f_{1,\infty},Z_{2})$, and $h_{top}^{WB}(f_{1,\infty},Z_{1})\leq h_{top}^{WB}(f_{1,\infty},Z_{2})$;
\item[(b)]
if $Z=\bigcup_{i\in I}Z_{i}$, with $I$ at most countable, then
\begin{itemize}
\item[(1)]
$P_{f_{1,\infty}}^{B}(Z,\psi)=\sup_{i\in I}P_{f_{1,\infty}}^{B}(Z_{i},\psi)$,
\item[(2)]
$P_{f_{1,\infty}}^{\mathcal{W}}(Z,\psi)=\sup_{i\in I}P_{f_{1,\infty}}^{\mathcal{W}}(Z_{i},\psi)$,
\item[(3)]
$h_{top}^{B}(f_{1,\infty},Z)=\sup_{i\in I}h_{top}^{B}(f_{1,\infty},Z_{i})$,
\item[(3)]
$h_{top}^{WB}(f_{1,\infty},Z)=\sup_{i\in I}h_{top}^{WB}(f_{1,\infty},Z_{i})$.
\end{itemize}
\end{enumerate}
\end{proposition}
In the following lemma, we study some properties of the functions $P_{\mu,f_{1,\infty}}(x,\psi)$, $\underline{h}_{\mu}(f_{1,\infty},x)$, $P_{\mu,f_{1,\infty}}(X,\psi)$, and $\underline{h}_{\mu}(f_{1,\infty})$ given by equations (\ref{eqq5}), (\ref{new1}), (\ref{eqq6}), and (\ref{new2}), respectively.

\begin{lemma}
Let $(X, f_{1,\infty})$ be an NDS on a compact metric space $(X,d)$, $\psi\in\mathcal{C}(X,\mathbb{R})$, $x\in X$ and $\mu\in\mathcal{M}(X)$. Then, $P_{\mu,f_{1,\infty}}(x,\psi)$ is integrable and $P_{\mu,f_{1,\infty}}(X,\psi)$ is well defined. In particular, $\underline{h}_{\mu}(f_{1,\infty},x)$ is integrable and $\underline{h}_{\mu}(f_{1,\infty})$ is well defined.
\end{lemma}
\begin{proof}
Since the negative part of $P_{\mu,f_{1,\infty}}(x,\psi)$ is finite, it is sufficient to show that the function
$x\mapsto\mu(B_{n}(x,\epsilon))$ is measurable for every $n\in\mathbb{N}$ and $\epsilon>0$. To see this, we shall show that $A_{c}:=\{x:\mu(B_{n}(x,\epsilon))\leq\epsilon\}$ is measurable for every $c\in\mathbb{R}$. Suppose $\{x_{m}\}$ is a sequence in $A_{c}$ with $\lim_{m\to\infty}x_{m}\to x_{0}$. Given $y\in B_{n}(x_{0},\epsilon)$, we have
\begin{equation*}
d(f_{1}^{i}(x_{m}),f_{1}^{i}(y))\leq d(f_{1}^{i}(x_{m}),f_{1}^{i}(x_{0}))+d(f_{1}^{i}(x_{0}),f_{1}^{i}(y))\ \ \text{for all}\ m\in\mathbb{N}\ \text{and}\ 0\leq i<n.
\end{equation*}
Then
\begin{equation*}
B_{n}(x_{0},\epsilon)\subseteq\liminf_{m\to\infty}B_{n}(x_{m},\epsilon)=\bigcup_{m=1}^{\infty}\bigcap_{k=m}^{\infty}B_{n}(x_{k},\epsilon).
\end{equation*}
Hence
\begin{equation*}
\mu(B_{n}(x_{0},\epsilon))\leq\mu\bigg(\bigcup_{m=1}^{\infty}\bigcap_{k=m}^{\infty}B_{n}(x_{k},\epsilon)\bigg)\leq\liminf_{m\to\infty}\mu(B_{n}(x_{m},\epsilon)).
\end{equation*}
It follows that $x_{0}\in A_{c}$, which implies that $A_{c}$ is closed and the function
$x\mapsto\mu(B_{n}(x,\epsilon))$ is measurable.
\end{proof}
We restate \cite[Lemma 1]{MJHWZY} as follows which is much like the classical covering lemma, and its proof in our setting is also true by using the dynamical balls.
\begin{lemma}\label{lemma1}
Let $(X, f_{1,\infty})$ be an NDS on a compact metric space $(X,d)$, $r>0$ and $\mathcal{B}(r):=\{B_{n}(x,r): x\in X, n\in\mathbb{N}\}$. For any family $\mathcal{F}\subseteq\mathcal{B}(r)$, there exists a (not necessarily countable) subfamily $\mathcal{G}\subseteq\mathcal{F}$ consisting of disjoint dynamical balls such that
\begin{equation*}
\bigcup_{B\in\mathcal{F}}B\subseteq\bigcup_{B_{n}(x,r)\in\mathcal{G}}B_{n}(x,3r).
\end{equation*}
\end{lemma}
Now, we prove Theorem A which shows that the Pesin-Pitskel topological pressure (weighted topological pressure) is determined by the measure-theoretic pressure of Borel probability measures, which extends \cite[Theorem A]{TXCWCZY} for NDSs. Then, we prove Corollary B which shows that the Bowen topological entropy (weighted Bowen topological entropy) is determined by the measure-theoretic lower entropy of Borel probability measures, which extends \cite[Theorem 1]{MJHWZY} for NDSs.

\textbf{Proof of Theorem A and Corollary B}. First, we prove part (1) of Theorem A for Pesin-Pitskel topological pressure. For a fixed $\gamma>0$, since $P_{\mu,f_{1,\infty}}(x,\psi)\leq s$ for all $x\in Z$, we have $Z=\cup_{m=1}^{\infty}Z_{m}$, where
\begin{equation*}
Z_{m}:=\bigg\{x\in Z:\liminf_{n\to\infty}\dfrac{-\log\mu(B_{n}(x,\epsilon))+S_{1,n}\psi(x)}{n}<s+\gamma\ \text{for all}\ \epsilon\in\bigg(0,\frac{1}{m}\bigg)\bigg\}.
\end{equation*}
Now, fix $m\geq 1$ and $\epsilon\in(0,\frac{1}{3m})$. For each $x\in Z_{m}$, there exists a strictly increasing sequence $\{n_{j}(x)\}_{j=1}^{\infty}$ such that
\begin{equation*}
\mu(B_{n_{j}(x)}(x,\epsilon))\geq\text{e}^{-(s+\gamma)n_{j}(x)+S_{1,n_{j}(x)}\psi(x)}\ \ \text{for all}\ j\geq 1.
\end{equation*}
Given $N\geq 1$, let
\begin{equation*}
\mathcal{F}_{N}:=\{B_{n_{j}(x)}(x,\epsilon): x\in Z_{m}\ \text{and}\ n_{j}(x)\geq N\}.
\end{equation*}
Then $Z_{m}\subseteq\cup_{B\in\mathcal{F}_{N}}B$.
Since $\mu$ is a probability measure and $\mu(B)>0$ for every $B\in\mathcal{F}_{N}$, by Lemma \ref{lemma1}, there exists a finite or countable subfamily 
$\mathcal{G}_{N}=\{B_{n_{i}}(x_{i},\epsilon)\}_{i\in I}\subseteq\mathcal{F}_{_{N}}$ consisting of pairwise  disjoint dynamical balls such that 
\begin{equation*}
Z_{m}\subseteq\bigcup_{i\in I}B_{n_{i}}(x_{i},3\epsilon)\ \ \text{and}\ \ \mu(B_{n_{i}}(x_{i},\epsilon))\geq \text{e}^{-(s+\gamma)n_{i}+S_{1,n_{i}}\psi(x_{i})}\ \ \text{for all}\ i\in I.
\end{equation*}
Consequently
\begin{equation*}
\mathcal{M}_{f_{1,\infty}}(N,s+\gamma,3\epsilon,Z_{m},\psi)\leq\sum_{i\in I}\text{e}^{-(s+\gamma)n_{i}+S_{1,n_{i}}\psi(x_{i})}\leq\sum_{i\in I}\mu(B_{n_{i}}(x_{i},\epsilon))\leq 1.
\end{equation*}
When $\epsilon\to 0$, by Theorem \ref{theorem3}, this implies 
\begin{equation*}
P_{f_{1,\infty}}^{B}(Z_{m},\psi)\leq s+\gamma\ \text{for all}\ m\geq 1.
\end{equation*}
Hence, by part (b) of Proposition \ref{theorem2}, we get
\begin{equation*}
P_{f_{1,\infty}}^{B}(Z,\psi)=\sup_{m\geq 1}P_{f_{1,\infty}}^{B}(Z_{m},\psi)\leq s+\gamma.
\end{equation*}
Therefore, $P_{f_{1,\infty}}^{B}(Z,\psi)\leq s$ since $\gamma> 0$ is arbitrary.

Now, we prove part (2) of Theorem A for Pesin-Pitskel topological pressure. Fix $\gamma>0$. For each $m\geq 1$, put
\begin{equation*}
Z_{m}:=\bigg\{x\in Z:\liminf_{n\to\infty}\dfrac{-\log\mu(B_{n}(x,\epsilon))+S_{1,n}\psi(x)}{n}>s-\gamma\ \text{for all}\ \epsilon\in\bigg(0,\frac{1}{m}\bigg]\bigg\}.
\end{equation*}
Since $\dfrac{-\log\mu(B_{n}(x,\epsilon))+S_{1,n}\psi(x)}{n}$ increases when $\epsilon$ decreases, it follows that
\begin{equation*}
Z_{m}=\bigg\{x\in Z:\liminf_{n\to\infty}\dfrac{-\log\mu(B_{n}(x,\epsilon))+S_{1,n}\psi(x)}{n}>s-\gamma\ \ \text{for}\ \ \epsilon=\frac{1}{m}\bigg\}.
\end{equation*}
Then $Z_{m}\subseteq Z_{m+1}$ and $Z=\bigcup_{m=1}^{\infty} Z_{m}$. So by the continuity of the measure, we have
\begin{equation*}
\lim_{m\to\infty}\mu(Z_{m})=\mu(Z).
\end{equation*}
Take $M\geq 1$ with $\mu(Z_{M})>\dfrac{1}{2}\mu(Z)$. For every $N\geq 1$, put
\begin{eqnarray*}
Z_{M,N}
&:= & \bigg\{x\in Z_{M}:\dfrac{-\log\mu(B_{n}(x,\epsilon))+S_{1,n}\psi(x)}{n}>s-\gamma\ \text{for all}\ n\geq N\ \text{and}\ \epsilon\in\bigg(0,\frac{1}{M}\bigg]\bigg\}\\
&= & \bigg\{x\in Z_{M}:\dfrac{-\log\mu(B_{n}(x,\epsilon))+S_{1,n}\psi(x)}{n}>s-\gamma\ \text{for all}\ n\geq N\ \text{and}\ \epsilon=\frac{1}{M}\bigg\}.
\end{eqnarray*}
Thus $Z_{M,N}\subseteq Z_{M,N+1}$ and $Z_{M}=\bigcup_{N=1}^{\infty} Z_{M,N}$. Again, we can find $N^{*}\geq 1$ such that $\mu(Z_{M,N^{*}})>\dfrac{1}{2}\mu(Z_{M})>0$. For every $x\in Z_{M,N^{*}}$, $n\geq N^{*}$ and $0<\epsilon<\frac{1}{M}$, we have
\begin{equation*}
\mu(B_{n}(x,\epsilon))\leq\text{e}^{-(s-\gamma)n+S_{1,n}\psi(x)}.
\end{equation*}
Suppose $\mathcal{F}=\{B_{n_{i}}(y_{i},\frac{\epsilon}{2})\}_{i\geq 1}$ is an open cover of $Z_{M,N^{*}}$ such that $Z_{M,N^{*}}\subseteq\bigcup_{i=1}^{\infty}B_{n_{i}}(y_{i},\frac{\epsilon}{2})$
and
\begin{equation*}
Z_{M,N^{*}}\cap B_{n_{i}}\bigg(y_{i},\frac{\epsilon}{2}\bigg)\neq\emptyset, n_{i}\geq N^{*}\ \text{for all}\ i\geq 1\ \text{and}\ 0<\epsilon<\frac{1}{M}.
\end{equation*}
For each $i\geq 1$, there exists $x_{i}\in Z_{M,N^{*}}\cap B_{n_{i}}(y_{i},\frac{\epsilon}{2})$. Hence, by the triangle inequality, we have
\begin{equation*}
B_{n_{i}}\bigg(y_{i},\frac{\epsilon}{2}\bigg)\subseteq B_{n_{i}}(x_{i},\epsilon).
\end{equation*}
This implies
\begin{eqnarray*}
\sum_{i\geq 1}\text{e}^{-(s-\gamma)n_{i}+S_{1,n_{i}}\psi(y_{i},\frac{\epsilon}{2})}
&\geq & \sum_{i\geq 1}\text{e}^{-(s-\gamma)n_{i}+S_{1,n_{i}}\psi(x_{i})}\\
&\geq & \sum_{i\geq 1} \mu(B_{n_{i}}(x_{i},\epsilon))\\
&\geq & \mu(Z_{M,N^{*}})>0.
\end{eqnarray*}
Therefore, 
\begin{equation*}
M_{f_{1,\infty}}(N,s-\gamma,\frac{\epsilon}{2},Z_{M,N^{*}},\psi)\geq \mu(Z_{M,N^{*}})>0.
\end{equation*}
This together with part (a) of Proposition \ref{theorem2} imply that $P_{f_{1,\infty}}^{B}(Z,\psi)\geq P_{f_{1,\infty}}^{B}(Z_{M,N^{*}},\psi)\geq s-\gamma$. Since $\gamma$ is arbitrary, we get $P_{f_{1,\infty}}^{B}(Z,\psi)\geq s$ which completes the proof of Theorem A for Pesin-Pitskel topological pressure. 

The proof of Theorem A for weighted topological pressure is a direct consequence of Proposition \ref{theorem5} which shows that $P_{f_{1,\infty}}^{B}(Z,\psi)=P_{f_{1,\infty}}^{\mathcal{W}}(Z,\psi)$.

On the other hand, for the proof of Corollary B, it is enough to take $\psi=0$ in the proof of Theorem A.
\section{Proof of theorem C and corollary D}\label{section4}
In this section, we manage the proof of Theorem C which is a Variational principle for topological pressure which links the Pesin-Pitskel topological pressure (weighted topological pressure) on an arbitrary nonempty compact subset to the measure-theoretic pressure of Borel probability measures for NDSs, this extends \cite[Theorem B, part (1)]{TXCWCZY}. Then, we prove Corollary D that is a variational principle for topological entropy which links the Bowen topological entropy (weighted Bowen topological entropy) on an arbitrary nonempty compact subset to the measure-theoretic lower entropy of Borel probability measures for NDSs, which extends \cite[Theorem 1.2, part (1)]{DFWH}.
In Proposition \ref{theorem5}, we prove that the Pesin-Pitskel topological pressure is equal to the weighted topological pressure, i.e., $P_{f_{1,\infty}}^{B}(Z,\psi)=P_{f_{1,\infty}}^{\mathcal{W}}(Z,\psi)$. To do this, we need the following lemma which is called the \emph{Vitali covering lemma} (see \cite[Theorem 2.1]{MP}) and employ the methods used in \cite{DFWH}.
\begin{lemma}\label{lemma2}
Let $(X,d)$ be a compact metric space and $\mathcal{B}=\{B(x_{i},r_{i})\}_{i\in\mathcal{I}}$
be a family of closed (or open) balls in $X$. Then there exists a finite or countable subfamily
$\mathcal{B}^{\prime}=\{B(x_{i},r_{i})\}_{i\in\mathcal{I}^{\prime}}$ of pairwise disjoint balls in $\mathcal{B}$ such that
\begin{equation*}
\bigcup_{B\in\mathcal{B}}B\subseteq\bigcup_{B(x_{i},r_{i})\in\mathcal{B}^{\prime}}B(x_{i},5r_{i}).
\end{equation*}
\end{lemma}
\begin{proposition}\label{theorem5}
Let $(X, f_{1,\infty})$ be an NDS on a compact metric space $(X,d)$, $Z$ be a subset of $X$ and $\psi\in\mathcal{C}(X,\mathbb{R})$. Then $P_{f_{1,\infty}}^{B}(Z,\psi)=P_{f_{1,\infty}}^{\mathcal{W}}(Z,\psi)$.
\end{proposition}
\begin{proof}
Taking $c_{i}=1$ in the equation (\ref{eqq7}), we have $\mathcal{W}_{f_{1,\infty}}(N,\alpha,\epsilon,Z,\psi)\leq M_{f_{1,\infty}}(N,\alpha,\epsilon,Z,\psi)$, for every $N\in\mathbb{N}$, that implies $P_{f_{1,\infty}}^{\mathcal{W}}(Z,\psi)\leq P_{f_{1,\infty}}^{B}(Z,\psi)$. To prove the opposite inequality, it suffices to show that
\begin{equation*}
\mathcal{M}_{f_{1,\infty}}(N,\alpha+\delta,6\epsilon,Z,\psi)\leq\mathcal{W}_{f_{1,\infty}}(N,\alpha,\epsilon,Z,\psi)
\end{equation*}
for any $\epsilon,\delta>0$, and $\alpha\in\mathbb{R}$, when $N$ is large enough.

For any $\delta>0$, there exists $N\geq 2$ such that $n^{2}\text{e}^{-n\delta}\leq 1$ for all $n\geq N$. Let $\mathcal{B}=\{(B_{n_{i}}(x_{i},\epsilon),c_{i})\}_{i\in\mathcal{I}}$ be a collection so that $\mathcal{I}\subseteq\mathbb{N}$, $0<c_{i}<\infty$, $n_{i}\geq N$ for all $i$, and
\begin{equation}\label{eqq8}
\sum_{i\in\mathcal{I}}c_{i}\chi_{B_{i}}\geq\chi_{Z},
\end{equation}
where $B_{i}:=B_{n_{i}}(x_{i},\epsilon)$. We shall prove that
\begin{equation}\label{eqq9}
\mathcal{M}_{f_{1,\infty}}(N,\alpha+\delta,6\epsilon,Z,\psi)\leq\sum_{i\in\mathcal{I}}c_{i}\text{e}^{-\alpha n_{i}+S_{1,n_{i}}\psi(x_{i},\epsilon)},
\end{equation}
which yields
\begin{equation*}
\mathcal{M}_{f_{1,\infty}}(N,\alpha+\delta,6\epsilon,Z,\psi)\leq\mathcal{W}_{f_{1,\infty}}(N,\alpha,\epsilon,Z,\psi).
\end{equation*}

Let $\mathcal{I}_{n}:=\{i\in\mathcal{I}:n_{i}=n\}$ and $\mathcal{I}_{n,k}:=\{i\in\mathcal{I}_{n}:i\leq k\}$ for any $n\geq N$ and $k\in\mathbb{N}$. For simplicity, we write $B_{i}:=B_{n_{i}}(x_{i},\epsilon)$ and $5B_{i}:=B_{n_{i}}(x_{i},5\epsilon)$ for $i\in\mathcal{I}$. Without loss of generality, we assume $B_{i}\neq B_{j}$ for $i\neq  j$. For $t>0$, put
\begin{equation*}
Z_{n,t}:=\bigg\{x\in Z: \sum_{i\in\mathcal{I}_{n}} c_{i}\chi_{B_{i}}(x)>t\bigg\}\ \ \text{and}\ \ Z_{n,k,t}:=\bigg\{x\in Z: \sum_{i\in\mathcal{I}_{n,k}} c_{i}\chi_{B_{i}}(x)>t\bigg\}.
\end{equation*}
Then the conclusion follows from the following three steps.

\textbf{Step 1}. 
For every $n\geq N$, $k\in\mathbb{N}$ and $t>0$, there exists a finite subset $\mathcal{J}_{n,k,t}$ of $\mathcal{I}_{n,k}$ such that the balls $B_{i}$ with $i\in\mathcal{J}_{n,k,t}$ are pairwise disjoint, $Z_{n,k,t}\subseteq\bigcup_{i\in\mathcal{J}_{n,k,t}}5B_{i}$ and
\begin{equation*}
\sum_{i\in\mathcal{J}_{n,k,t}}\text{e}^{-\alpha n+S_{1,n}\psi(x_{i},\epsilon)}\leq\dfrac{1}{t}\sum_{i\in\mathcal{I}_{n,k}}c_{i}\text{e}^{-\alpha n+S_{1,n}\psi(x_{i},\epsilon)}.
\end{equation*}
$Proof\ of\ step\ 1$. By approximating the constants $c_{i}$ from above, we may assume that each $c_{i}$ is a positive rational (note that each $\mathcal{I}_{n,k}$ is finite). Then multiplying with a common denominator on both sides of the inequality
in $Z_{n,k,t}$, we may assume that each $c_{i}$ is a positive integer. Let $m$ be the least integer with
$m\geq t$. Denote $\mathcal{C}_{0}=\mathcal{B}_{0}:=\{B_{i}: i\in\mathcal{I}_{n,k}\}$ and define $u:\mathcal{B}_{0}\to\mathbb{N}$ by $u(B_{i})=c_{i}$. We define by
induction integer-valued functions $\nu_{0},\nu_{1},\ldots,\nu_{m}$ on $\mathcal{B}_{0}$ and subfamilies $\mathcal{B}_{1},\mathcal{B}_{2},\ldots,\mathcal{B}_{m}$ of $\mathcal{B}_{0}$ starting
with $\nu_{0}=u$. Using Lemma \ref{lemma2} (in which we take the metric $d_{n}$ instead of $d$) we find a pairwise disjoint subfamily $\mathcal{B}_{1}$ of $\mathcal{B}_{0}$ such that $\bigcup_{B\in\mathcal{B}_{0}}B\subseteq\bigcup_{B\in\mathcal{B}_{1}}5B$, and hence $Z_{n,k,t}\subseteq\bigcup_{B\in\mathcal{B}_{1}}5B$. Define
\begin{equation*}
\nu_{1}(B) = \left\{
\begin{array}{rl}
\nu_{0}(B)-1 & \text{if } B\in\mathcal{B}_{1},\\
\nu_{0}(B) & \text{if } B\in\mathcal{B}_{0}\setminus\mathcal{B}_{1}.
\end{array} \right.
\end{equation*}
Let $\mathcal{C}_{1}:=\{B\in\mathcal{B}_{0}: \nu_{1}(B)\geq 1\}$. Since the subfamily $\mathcal{B}_{1}$ is pairwise disjoint, $Z_{n,k,t}\subseteq\{x: \Sigma_{B\in\mathcal{B}_{0},x\in B}\nu_{1}(B)\geq m-1\}$ which implies that every $x\in Z_{n,k,t}$ belongs to some ball $B\in\mathcal{B}_{0}$
with $\nu_{1}(B)\geq 1$. Thus $Z_{n,k,t}\subseteq\bigcup_{B\in\mathcal{C}_{1}} B$. By Lemma \ref{lemma2}, we pick a pairwise disjoint subfamily $\mathcal{B}_{2}$ of $\mathcal{C}_{1}$ such that $\bigcup_{B\in\mathcal{C}_{1}}B\subseteq\bigcup_{B\in\mathcal{B}_{2}}5B$, and hence $Z_{n,k,t}\subseteq\bigcup_{B\in\mathcal{B}_{2}}5B$. Define
\begin{equation*}
\nu_{2}(B) = \left\{
\begin{array}{rl}
\nu_{1}(B)-1 & \text{if } B\in\mathcal{B}_{2},\\
\nu_{1}(B) & \text{if } B\in\mathcal{B}_{0}\setminus\mathcal{B}_{2}.
\end{array} \right.
\end{equation*}
Let $\mathcal{C}_{2}:=\{B\in\mathcal{B}_{0}: \nu_{2}(B)\geq 1\}$. Then 
$Z_{n,k,t}\subseteq\{x: \Sigma_{B\in\mathcal{B}_{0},x\in B}\nu_{2}(B)\geq m-2\}$ which implies
that every $x\in Z_{n,k,t}$ belongs to some $B\in\mathcal{B}_{0}$ with $\nu_{2}(B)\geq 1$. Thus  $Z_{n,k,t}\subseteq\bigcup_{B\in\mathcal{C}_{2}}B$.
Repeating the above process for $j=1,2,\ldots,m$, we get subfamilies $\mathcal{C}_{j-1}$ and $\mathcal{B}_{j}$ of $\mathcal{B}_{0}$ such that each $\mathcal{B}_{j}$ consists of pairwise disjoint elements, $\mathcal{B}_{j}\subseteq\mathcal{C}_{j-1}$, $Z_{n,k,t}\subseteq\bigcup_{B\in\mathcal{B}_{j}}5B$, $Z_{n,k,t}\subseteq\bigcup_{B\in\mathcal{C}_{j-1}}B$ and
\begin{equation*}
\nu_{j}(B) = \left\{
\begin{array}{rl}
\nu_{j-1}(B)-1 & \text{if } B\in\mathcal{B}_{j},\\
\nu_{j-1}(B) & \text{if } B\in\mathcal{B}_{0}\setminus\mathcal{B}_{j}.
\end{array} \right.
\end{equation*}
Therefore, 
\begin{eqnarray*}
\sum_{j=1}^{m}\sum_{B_{n}(x_{i},\epsilon)\in\mathcal{B}_{j}}\text{e}^{-\alpha n+S_{1,n}\psi(x_{i},\epsilon)}
&= & \sum_{j=1}^{m}\sum_{B_{n}(x_{i},\epsilon)\in\mathcal{B}_{j}}(\nu_{j-1}(B_{n}(x_{i},\epsilon))-\nu_{j}(B_{n}(x_{i},\epsilon)))\text{e}^{-\alpha n+S_{1,n}\psi(x_{i},\epsilon)}\\
&\leq & \sum_{B_{n}(x_{i},\epsilon)\in\mathcal{B}_{0}}\sum_{j=1}^{m}(\nu_{j-1}(B_{n}(x_{i},\epsilon))-\nu_{j}(B_{n}(x_{i},\epsilon)))\text{e}^{-\alpha n+S_{1,n}\psi(x_{i},\epsilon)}\\
&\leq & \sum_{B_{n}(x_{i},\epsilon)\in\mathcal{B}_{0}}u(B_{n}(x_{i},\epsilon))\text{e}^{-\alpha n+S_{1,n}\psi(x_{i},\epsilon)}\\
&= & \sum_{i\in\mathcal{I}_{n,k}}c_{i}\text{e}^{-\alpha n+S_{1,n}\psi(x_{i},\epsilon)}.
\end{eqnarray*}
Now, choose $j_{0}\in\{1,2,\ldots,m\}$ such that $\sum_{B_{n}(x_{i},\epsilon)\in\mathcal{B}_{j_{0}}}\text{e}^{-\alpha n+S_{1,n}\psi(x_{i},\epsilon)}$ is the smallest. Then we have
\begin{eqnarray*}
\sum_{B_{n}(x_{i},\epsilon)\in\mathcal{B}_{j_{0}}}\text{e}^{-\alpha n+S_{1,n}\psi(x_{i},\epsilon)}
&\leq & \dfrac{1}{m}\sum_{i\in\mathcal{I}_{n,k}}c_{i}\text{e}^{-\alpha n+S_{1,n}\psi(x_{i},\epsilon)}\\
&\leq & \dfrac{1}{t}\sum_{i\in\mathcal{I}_{n,k}}c_{i}\text{e}^{-\alpha n+S_{1,n}\psi(x_{i},\epsilon)}.
\end{eqnarray*}
Hence, it is enough to take $\mathcal{J}_{n,k,t}=\{i\in\mathcal{I}_{n,k}: B_{i}\in\mathcal{B}_{j_{0}}\}$.

\textbf{Step 2}. 
For every $n\geq N$ and $t>0$, we have
\begin{equation*}
\mathcal{M}_{f_{1,\infty}}(N,\alpha+\delta,6\epsilon,Z_{n,t},\psi)\leq\dfrac{1}{n^{2}t}\sum_{i\in\mathcal{I}_{n}}c_{i}\text{e}^{-\alpha n+S_{1,n}\psi(x_{i},\epsilon)}.
\end{equation*}
$Proof\ of\ step\ 2$. For $Z_{n,t}=\emptyset$ the inequality is obvious. Assume $Z_{n,t}\neq\emptyset$. Then $Z_{n,k,t}\neq\emptyset$ when $k$ is large enough, because of $Z_{n,k,t}\uparrow Z_{n,t}$. Let $\mathcal{J}_{n,k,t}$ be the sets constructed in Step 1. Then $\mathcal{J}_{n,k,t}\neq\emptyset$ when $k$ is sufficiently large. Define $E_{n,k,t}:=\{x_{i}: i\in\mathcal{J}_{n,k,t}\}$. 
Note that the family of all nonempty compact subsets of $X$ is compact with respect to the Hausdorff distance (see
Federer \cite[2.10.21]{FH}). Therefore, there exists a subsequence $\{k_{j}\}$ in $\mathbb{N}$ and a non-empty compact
set $E_{n,t}\subseteq X$ such that $E_{n,k_{j},t}$ converges to $E_{n,t}$ in the Hausdorff distance as $j\to\infty$. Note that any two points in $E_{n,k,t}$ have a distance not less than $\epsilon$ with respect to $d_{n}$, which implies $E_{n,t}$ has the same property. Thus $E_{n,t}$ is a finite set. This yields that $|E_{n,k_{j},t}|=|E_{n,t}|$ when $j$ is large enough. So, when $j$ is sufficiently large, we have
\begin{equation*}
Z_{n,k_{j},t}\subseteq\bigcup_{i\in\mathcal{J}_{n,k_{j},t}}5B_{i}=\bigcup_{x\in E_{n,k_{j},t}}B_{n}(x,5\epsilon)\subseteq\bigcup_{x\in E_{n,t}}B_{n}(x,5.5\epsilon).
\end{equation*}
Thus $Z_{n,t}\subseteq\bigcup_{x\in E_{n,t}}B_{n}(x,6\epsilon)$. Since $|E_{n,k_{j},t}|=|E_{n,t}|$ when $j$ is large enough, we have
\begin{eqnarray*}
\sum_{x\in E_{n,t}}\text{e}^{-\alpha n+S_{1,n}\psi(x)}
&\leq & \sum_{x\in E_{n,k_{j},t}}\text{e}^{-\alpha n+S_{1,n}\psi(x,\epsilon)}\\
&\leq & \dfrac{1}{t}\sum_{i\in\mathcal{I}_{n}}c_{i}\text{e}^{-\alpha n+S_{1,n}\psi(x_{i},\epsilon)}.
\end{eqnarray*}
Therefore,
\begin{eqnarray*}
\mathcal{M}_{f_{1,\infty}}(N,\alpha+\delta,6\epsilon,Z_{n,t},\psi)
&\leq & \sum_{x\in E_{n,t}}\text{e}^{-(\alpha+\delta)n+S_{1,n}\psi(x)}\\
&\leq & \dfrac{1}{\text{e}^{n\delta}t}\sum_{i\in\mathcal{I}_{n}}c_{i}\text{e}^{-\alpha n+S_{1,n}\psi(x_{i},\epsilon)}\\
&\leq & \dfrac{1}{n^{2}t}\sum_{i\in\mathcal{I}_{n}}c_{i}\text{e}^{-\alpha n+S_{1,n}\psi(x_{i},\epsilon)}.
\end{eqnarray*}

\textbf{Step 3}. For any $t\in(0,1)$, we have the following inequality that implies the inequality (\ref{eqq9}):
\begin{equation*}
\mathcal{M}_{f_{1,\infty}}(N,\alpha+\delta,6\epsilon,Z,\psi)\leq\dfrac{1}{t}\sum_{i\in\mathcal{I}}c_{i}\text{e}^{-\alpha n_{i}+S_{1,n_{i}}\psi(x_{i},\epsilon)}.
\end{equation*}
$Proof\ of\ step\ 3$.
Fix $t\in(0,1)$. Then, it follows that $Z\subseteq\bigcup_{n=N}^{\infty}Z_{n,n^{-2}t}$ from the inequality (\ref{eqq8}). Now, we prove the following inequality
\begin{equation}\label{eqq10}
\mathcal{M}_{f_{1,\infty}}(N,\alpha+\delta,6\epsilon,\bigcup_{n=N}^{\infty}Z_{n,n^{-2}t},\psi)\leq\sum_{n=N}^{\infty}\mathcal{M}_{f_{1,\infty}}(N,\alpha+\delta,6\epsilon,Z_{n,n^{-2}t},\psi).
\end{equation}
Given $\gamma>0$, for every $n\geq N$, since
\begin{equation*}
\mathcal{M}_{f_{1,\infty}}(N,\alpha+\delta,6\epsilon,Z_{n,n^{-2}t},\psi)=\inf\bigg\{\sum_{i}\text{e}^{-(\alpha+\delta)n_{i}+S_{1,n_{i}}\psi(x_{i})}: Z_{n,n^{-2}t}\subseteq\bigcup_{i}B_{n_{i}}(x_{i},6\epsilon)\bigg\},
\end{equation*}
we can choose a sequence of subsets $\{B_{n_{j}}(x_{n_{j}},6\epsilon)\}_{j\in\mathbb{N}}$ such that $Z_{n,n^{-2}t}\subseteq\bigcup_{j\in\mathbb{N}}B_{n_{j}}(x_{n_{j}},6\epsilon)$ and
\begin{equation*}
\sum_{j\in\mathbb{N}}\text{e}^{-(\alpha+\delta)n_{j}+S_{1,n_{j}}\psi(x_{n_{j}})}\leq\mathcal{M}_{f_{1,\infty}}(N,\alpha+\delta,6\epsilon,Z_{n,n^{-2}t},\psi)+\dfrac{\gamma}{2^{n}}.
\end{equation*}
Then, it is clear that the union of sequences $\{B_{n_{j}}(x_{n_{j}},6\epsilon)\}_{j\in\mathbb{N}}$ for $n\geq N$ is a cover of $\bigcup_{n=N}^{\infty}Z_{n,n^{-2}t}$. This implies
\begin{eqnarray*}
\mathcal{M}_{f_{1,\infty}}(N,\alpha+\delta,6\epsilon,\bigcup_{n=N}^{\infty}Z_{n,n^{-2}t},\psi)
&\leq &\sum_{n=N}^{\infty}\sum_{j\in\mathbb{N}}\text{e}^{-(\alpha+\delta)n_{j}+S_{1,n_{j}}\psi(x_{n_{j}})}\\
&\leq & \sum_{n=N}^{\infty}\mathcal{M}_{f_{1,\infty}}(N,\alpha+\delta,6\epsilon,Z_{n,n^{-2}t},\psi)+\gamma.
\end{eqnarray*}
Since $\gamma$ is arbitrarily small, we get the inequality (\ref{eqq10}). Thus, by the Step 2, we have
\begin{eqnarray*}
\mathcal{M}_{f_{1,\infty}}(N,\alpha+\delta,6\epsilon,Z,\psi)
&\leq & \sum_{n=N}^{\infty}\mathcal{M}_{f_{1,\infty}}(N,\alpha+\delta,6\epsilon,Z_{n,n^{-2}t},\psi)\\
&\leq & \sum_{n=N}^{\infty}\dfrac{1}{t}\sum_{i\in\mathcal{I}_{n}}c_{i}\text{e}^{-\alpha n+S_{1,n}\psi(x_{i},\epsilon)}\\
&= &\dfrac{1}{t}\sum_{i\in\mathcal{I}}c_{i}\text{e}^{-\alpha n_{i}+S_{1,n_{i}}\psi(x_{i},\epsilon)},
\end{eqnarray*}
that implies the inequality (\ref{eqq9}), and so 
\begin{equation*}
\mathcal{M}_{f_{1,\infty}}(N,\alpha+\delta,6\epsilon,Z,\psi)\leq\mathcal{W}_{f_{1,\infty}}(N,\alpha,\epsilon,Z,\psi).
\end{equation*}
This completes the proof, by Theorem \ref{theorem3}.
\end{proof}
To prove Theorem C, we need the following dynamical Frostman's lemma which is an analogue of Feng and Huang's approximation. For completeness, we give its proof. Our arguments follow the proof of \cite[Lemma 3.4]{DFWH} which is adapted from Howroyd's elegant argument (see \cite[Theorem 2]{HJD} and \cite[Theorem 8.17]{MP}).
\begin{lemma}\label{lemma3}
Let $(X, f_{1,\infty})$ be an NDS on a compact metric space $(X,d)$, $Z$ be a nonempty compact subset of $X$ and $\psi\in\mathcal{C}(X,\mathbb{R})$. Let $\alpha\in\mathbb{R}$, $N\in\mathbb{N}$,
and $\epsilon>0$. Suppose that $c:=\mathcal{W}_{f_{1,\infty}}(N,\alpha,\epsilon,Z,\psi)>0$. Then there is  $\mu\in\mathcal{M}(X)$ such that $\mu(Z)=1$ and
\begin{equation*}
\mu(B_{n}(x,\epsilon))\leq\dfrac{1}{c}\text{e}^{-\alpha n+S_{1,n}\psi(x,\epsilon)},\ \ \ \forall x\in X,\ n\geq N.
\end{equation*}
\end{lemma}
\begin{proof}
Clearly $c<\infty$. We define a function $p$ on the Banach space $\mathcal{C}(X,\mathbb{R})$ of all continuous real-valued functions on $X$ equipped with the supremum norm by
\begin{equation*}
p(\phi):=\dfrac{1}{c}\mathcal{W}_{f_{1,\infty}}(N,\alpha,\epsilon,\chi_{Z}.\phi,\psi).
\end{equation*}

Let $\textbf{1}\in\mathcal{C}(X,\mathbb{R})$ denote the constant function $\textbf{1}(x)\equiv 1$. It is easy to verify that
\begin{itemize}
\item[(1)]
$p(\phi+\varphi)\leq p(\phi)+p(\varphi)$ for any $\phi,\varphi\in\mathcal{C}(X,\mathbb{R})$.
\item[(2)]
$p(t\phi)=t p(\phi)$ for any $t\geq 0$ and $\phi\in\mathcal{C}(X,\mathbb{R})$.
\item[(3)]
$p(\textbf{1})=1$, $0\leq p(\phi)\leq\|\phi\|_{\infty}$ for any $\phi\in\mathcal{C}(X,\mathbb{R})$, and $p(\phi)=0$ for any $\phi\in\mathcal{C}(X,\mathbb{R})$ with $\phi\leq 0$.
\end{itemize}
By the Hahn-Banach theorem, we can extend the linear functional $t\mapsto tp(\textbf{1})$, $t\in\mathbb{R}$, from the subspace of the constant functions to a linear functional $L:\mathcal{C}(X,\mathbb{R})\to\mathbb{R}$ satisfying
\begin{equation*}
L(\textbf{1})=p(\textbf{1})=1\ \ \text{and}\ \ -p(-\phi)\leq L(\phi)\leq p(\phi)\ \ \text{for any}\ \ \phi\in\mathcal{C}(X,\mathbb{R}). 
\end{equation*}
If $\phi\in\mathcal{C}(X,\mathbb{R})$ with $\phi\geq 0$, then $p(-\phi)=0$ and so $L(\phi)\geq 0$. Hence combining the fact that $L(\textbf{1})=1$, we can use the Riesz representation theorem to find a Borel probability measure $\mu$ on $X$ such that $L(\phi)=\int\phi\ d\mu$ for $\phi\in\mathcal{C}(X,\mathbb{R})$.

Now we show that $\mu(Z)=1$. To see this, for any compact set $E\subseteq X\setminus Z$, by the Urysohn lemma there is $\phi\in\mathcal{C}(X,\mathbb{R})$ such that $0\leq\phi\leq 1$, $\phi(x)=1$ for $x\in E$ and $\phi(x)=0$ for $x\in Z$. It follows that $\chi_{Z}.\phi\equiv 0$ and thus $p(\phi)=0$. Hence $\mu(E)\leq L(\phi)\leq p(\phi)=0$. Since $\mu$ is regular, we have $\mu(X\setminus Z)=0$. This means that $\mu(Z)=1$.

Now, we show that
\begin{equation*}
\mu(B_{n}(x,\epsilon))\leq\dfrac{1}{c}\text{e}^{-\alpha n+S_{1,n}\psi(x,\epsilon)},\ \ \ \forall x\in X, n\geq N.
\end{equation*}
To see this, for any compact set $E\subseteq B_{n}(x,\epsilon)$, by the Urysohn lemma, there exists $\phi\in\mathcal{C}(X,\mathbb{R})$ such that $0\leq\phi\leq 1$, $\phi(x)=1$ for $x\in E$ and $\phi(x)=0$ for $x\in X\setminus B_{n}(x,\epsilon)$. Then $\mu(E)\leq L(\phi)\leq p(\phi)$. Since $\chi_{Z}.\phi\leq\chi_{B_{n}(x,\epsilon)}$ and $n\geq N$, we have $\mathcal{W}_{f_{1,\infty}}(N,\alpha,\epsilon,\chi_{Z}.\phi,\psi)\leq\text{e}^{-\alpha n+S_{1,n}\psi(x,\epsilon)}$ that implies $p(\phi)\leq\frac{1}{c}\text{e}^{-\alpha n+S_{1,n}\psi(x,\epsilon)}$ and $\mu(E)\leq\frac{1}{c}\text{e}^{-\alpha n+S_{1,n}\psi(x,\epsilon)}$. Using the regularity of $\mu$, we obtain $\mu(B_{n}(x,\epsilon))\leq\frac{1}{c}\text{e}^{-\alpha n+S_{1,n}\psi(x,\epsilon)}$.
\end{proof}
\textbf{Proof of Theorem C and Corollary D}.
Assume $\mu(Z)=1$, since $P_{\mu,f_{1,\infty}}(X,\psi)=\int P_{\mu,f_{1,\infty}}(x,\psi)\ d\mu(x)$, it follows that the set
\begin{equation*}
Z_{\delta}:=\{x\in Z: P_{\mu,f_{1,\infty}}(x,\psi)\geq P_{\mu,f_{1,\infty}}(X,\psi)-\delta\}
\end{equation*}
has positive $\mu$-measure for all $\delta>0$, i.e., $\mu(Z_{\delta})>0$ for all $\delta>0$. Thus, by part (a) of Proposition \ref{theorem2} and part (2) of Theorem A, we obtain
\begin{equation*}
P_{f_{1,\infty}}^{B}(Z,\psi)\geq P_{f_{1,\infty}}^{B}(Z_{\delta},\psi)\geq P_{\mu,f_{1,\infty}}(X,\psi)-\delta\ \ \text{for all}\ \ \delta>0.
\end{equation*}
This implies that
\begin{equation*}
P_{f_{1,\infty}}^{B}(Z,\psi)\geq\sup\{P_{\mu,f_{1,\infty}}(X,\psi):\mu\in\mathcal{M}(X), \mu(Z)=1\}.
\end{equation*}

Now, we show the converse inequality. Without loss of generality, we can assume that
$P_{f_{1,\infty}}^{B}(Z,\psi)\neq -\infty$, otherwise there is nothing to prove. By Proposition \ref{theorem5}, 
we have $P_{f_{1,\infty}}^{B}(Z,\psi)=P_{f_{1,\infty}}^{\mathcal{W}}(Z,\psi)$. Fix a small number $\beta>0$ and let $\alpha:=P_{f_{1,\infty}}^{B}(Z,\psi)-\beta$. Since
\begin{equation*}
\lim_{\epsilon\to 0}\liminf_{n\to\infty}\dfrac{1}{n}[S_{1,n}\psi(x)-S_{1,n}\psi(x,\epsilon)]=0
\end{equation*}
for all $x\in X$, we get
\begin{equation*}
\liminf_{n\to\infty}\dfrac{1}{n}[S_{1,n}\psi(x)-S_{1,n}\psi(x,\epsilon)]>-\beta
\end{equation*}
for all $x\in X$ and all sufficiently small $\epsilon>0$. Take such an $\epsilon>0$ and $N\in\mathbb{N}$ such that $c:=\mathcal{W}_{f_{1,\infty}}(N,\alpha,\epsilon,Z,\psi)>0$. By
Lemma \ref{lemma3}, there exists $\mu\in\mathcal{M}(X)$ with $\mu(Z)=1$ such that
\begin{equation*}
\mu(B_{n}(x,\epsilon))\leq\dfrac{1}{c}\text{e}^{-\alpha n+S_{1,n}\psi(x,\epsilon)}\ \ \text{for all}\ \ x\in X\ \ \text{and}\ \ n\geq N.
\end{equation*}
This implies that
\begin{equation*}
\liminf_{n\to\infty}\dfrac{-\log\mu(B_{n}(x,\epsilon))+S_{1,n}\psi(x,\epsilon)}{n}\geq\alpha\ \ \text{for all}\ \ x\in X.
\end{equation*}
Therefore, for all $x\in X$ we get
\begin{eqnarray*}
P_{\mu,f_{1,\infty}}(x,\psi)
&= & \lim_{\epsilon\to 0}\liminf_{n\to\infty}\dfrac{-\log\mu(B_{n}(x,\epsilon))+S_{1,n}\psi(x)}{n}\\
&\geq & \liminf_{n\to\infty}\dfrac{-\log\mu(B_{n}(x,\epsilon))+S_{1,n}\psi(x)}{n}\\
&= & \liminf_{n\to\infty}\dfrac{-\log\mu(B_{n}(x,\epsilon))+S_{1,n}\psi(x)+S_{1,n}\psi(x,\epsilon)-S_{1,n}\psi(x,\epsilon)}{n}\\
&\geq & \liminf_{n\to\infty}\dfrac{-\log\mu(B_{n}(x,\epsilon))+S_{1,n}\psi(x,\epsilon)}{n}+\liminf_{n\to\infty}\dfrac{1}{n}[S_{1,n}\psi(x)-S_{1,n}\psi(x,\epsilon)]\\
&\geq & \alpha-\beta=P_{f_{1,\infty}}^{B}(Z,\psi)-2\beta.
\end{eqnarray*}
Hence
\begin{equation*}
P_{\mu,f_{1,\infty}}(X,\psi)=\int P_{\mu,f_{1,\infty}}(x,\psi)\ d\mu(x)\geq P_{f_{1,\infty}}^{B}(Z,\psi)-2\beta.
\end{equation*}
Since $\beta$ is arbitrary, we get the desired inequality which completes the proof of Theorem C for Pesin-Pitskel topological pressure. 

The proof of Theorem C for weighted topological pressure is a direct consequence of Proposition \ref{theorem5} which shows that $P_{f_{1,\infty}}^{B}(Z,\psi)=P_{f_{1,\infty}}^{\mathcal{W}}(Z,\psi)$.  

On the other hand, for the proof of Corollary D, it is enough to take $\psi=0$ in the proof of Theorem C.
\section*{Acknowledgements}
The authors would like to thank the respectful referee for his/her comments on the manuscript.

\end{document}